\documentclass[a4paper,10pt]{article}
\usepackage[english]{babel}
\usepackage[latin1]{inputenc}
\usepackage[T1]{fontenc}
\usepackage{amsmath}
\usepackage{amssymb}
\usepackage{multirow}
\usepackage{alltt}
\usepackage{graphicx}
\usepackage{bbm}
\usepackage{amsthm}
\usepackage{mathrsfs}
\usepackage{textcomp}
\usepackage{fullpage}
\usepackage{amsrefs}
\usepackage[usenames,dvipsnames]{color}

\definecolor{darkblue}{RGB}{0,0,170}
\definecolor{brickred}{RGB}{200,0,0}
\usepackage[pdfauthor={Nicola Abatangelo}, hyperindex, linktocpage]{hyperref}
\hypersetup{colorlinks=true, linkcolor=darkblue, citecolor=brickred}

\newcommand{\R}{\mathbb{R}}
\newcommand{\N}{\mathbb{N}}
\newcommand{\C}{\mathcal{C}}
\newcommand{\A}{\mathcal{A}}
\newcommand{\T}{\mathcal{T}}

\newcommand{\eps}{\varepsilon}
\newcommand{\dist}{\hbox{dist}}
\newcommand{\Ds}{{\left(-\lapl\right)}^s}
\newcommand{\lapl}{\triangle}
\newcommand{\grad}{\triangledown\!}

\newtheoremstyle{mytheoremstyle} 
    {\topsep}                    
    {\topsep}                    
    {\itshape}                   
    {}                           
    {\bf\scshape}                   
    {.}                          
    {.5em}                       
    {\thmname{#1}\thmnumber{\oldstylenums{ #2}}\thmnote{: #3}}  

\theoremstyle{mytheoremstyle}
\newtheorem{defi}{definition}[]
\newtheorem{ex}[defi]{example}
\newtheorem{theo}[defi]{theorem}
\newtheorem{prop}[defi]{proposition}
\newtheorem{lem}[defi]{lemma}
\newtheorem{rmk}[defi]{remark}

\usepackage{titlesec}
\titleformat{\section}{\centering\Large\bf\scshape}{\oldstylenums{\thesection}.}{.4em}{}[]
\titleformat{\subsection}{\normalfont\large\bfseries}{\oldstylenums{\thesubsection}}{.4em}{}[]

\usepackage{tocloft}

\title{\bfseries\scshape\LARGE Very large solutions for the fractional Laplacian: towards a fractional Keller-Osserman condition}
\author{\large\scshape nicola abatangelo\footnote{Universit\'e Libre de Bruxelles, 
Avenue Franklin Roosevelt 50, 1050 Bruxelles, Belgium.}} 
\date{ }

\begin{document}

\maketitle

%

\begin{center}
\begin{minipage}{.9\textwidth}\small
{\bf   \scshape{abstract.}} We look for solutions of
\[
\Ds u+f(u)\ =\ 0\ \hbox{ in a bounded smooth domain $\Omega$},\ s\in(0,1),
\]
with a strong singularity at the boundary. 
In particular, we are interested in solutions which are $L^1(\Omega)$ and higher order
with respect to dist$(x,\partial\Omega)^{s-1}$. We provide
sufficient conditions for the existence of such a solution.
Roughly speaking, these functions are the real fractional counterpart
of {\it large solutions} in the classical setting.
\end{minipage}
\end{center}
%
\bigskip\bigskip

{\small
\noindent {\it Key words:} fractional Laplacian, large solutions, Keller-Osserman condition.\\
\noindent {\it 2010 MSC:} Primary: 35A01, 45K05, 35B40; Secondary: 35S15.
}

\section{introduction}

In the theory of semilinear elliptic equations,
functions solving
\begin{equation}\label{eq}
-\lapl u+ f(u)=0,\qquad\hbox{in some }\Omega\subseteq\R^N\hbox{ open and bounded}
\end{equation}
coupled with the boundary condition
\[
\lim_{x\rightarrow\partial\Omega}u(x)=+\infty
\]
are known as {\it boundary blow-up solutions} or {\it large solutions}.
There is a huge amount of bibliography dealing with this problem
which dates back to the seminal work of Bieberbach \cite{bieberbach}, for $N=2$ and $f(u)=e^u$.
Keller \cite{keller} and Osserman \cite{osserman} independently established a sufficient and necessary
condition on the nonlinear term $f$ for the existence of a boundary blow-up solution:
it takes the form
\begin{equation}\label{ko}
\int^{+\infty}\frac{dt}{\sqrt{F(t)}}\ <\ +\infty,\qquad\hbox{where }F'=f\geq 0
\end{equation}
and it is known as {\it Keller-Osserman condition}.
One can find these solutions with singular behaviour at the boundary in a number of applications:
for example, Loewner and Nirenberg \cite{ln} studied the case $f(u)=u^{(N+2)/(N-2)},\;N\geq3$,
which is strictly related to the {\it singular Yamabe problem} in conformal Geometry,
while Labutin \cite{labutin} completely characterized the class of sets $\Omega$
that admit a large solution for $f(u)=u^q,\,q>1,$ with capacitary methods 
inspired by the theory of {\it spatial branching processes}, that are particular stochastic processes;
see also the purely probabilistic works
by Le Gall \cite{legall} and Dhersin and Le Gall \cite{legall-dhersin} dealing with the particular case $q=2$.
\smallskip

In this paper we tackle equation \eqref{eq} when the Laplacian operator is replaced 
by one of its fractional powers. The fractional Laplacian $\Ds,\;s\in(0,1)$, 
is an integral nonlocal operator of fractional order
which admits different equivalent definitions, see e.g. Di Nezza, Palatucci and Valdinoci \cite{hitchhiker}: 
we will use the following
\begin{equation}\label{fr-lapl}
\Ds u(x)\ =\ \A(N,s)\,PV\int_{\R^N}\frac{u(x)-u(y)}{|x-y|^{N+2s}}\;dy\ =
\ \A(N,s)\,\lim_{\eps\downarrow0}\int_{\{|y-x|>\eps\}}\frac{u(x)-u(y)}{|x-y|^{N+2s}}\;dy
\end{equation}
where $\A(N,s)$ is a renormalizing positive constant. This operator generates\footnote[1]{Recall 
that $-\lapl$ is the infinitesimal generator of the Wiener process, 
modelling the Brownian motion.}
a Wiener process subordinated in time with an $s$-stable L\'evy process.
The Dirichlet problem related to $\Ds$ is of the form
\[
\left\lbrace\begin{aligned}
\Ds u &= f & & \hbox{in }\Omega \\
u &= g & & \hbox{in } \R^N\setminus\Omega
\end{aligned}\right.
\]
because the data have to take into account the nonlocal character of the operator.
Nevertheless in \cite{a} the author showed how this problem is ill-posed in a weak $L^1$ sense, of Stampacchia's sort,
unless a singular trace is prescribed at the boundary. A well-posed Dirichlet problem needs to deal with two conditions at the same time: namely,
if $d$ denotes the distance to the boundary $\partial\Omega$, it looks like 
\[
\left\lbrace\begin{aligned}
\Ds u &= f & & \hbox{in }\Omega \\
u &= g & & \hbox{in } \R^N\setminus\Omega \\
d^{1-s}u &= h & & \hbox{on }\partial\Omega
\end{aligned}\right.
\]
where the data satisfy the following assumptions
\[
\int_\Omega |f|\,d^s<+\infty,\quad\int_{\R^N\setminus\Omega}|g|d^{-s}\min\{1,d^{-N-s}\}<+\infty,\quad\|h\|_{L^\infty(\partial\Omega)}<+\infty.
\]
Further references in this direction are the recent works by Grubb \cites{grubb1,grubb2},
where also the regularity up to the boundary is investigated.
This means in particular that in the context of fractional Dirichlet problems there are solutions with an explosive behaviour at the boundary
as a result of a linear phenomenon: for instance the solutions to
\[
\left\lbrace\begin{aligned}
\Ds u &= 0 & & \hbox{in }B_1 \\
u(x) &= (|x|^2-1)^{-s/2} & & \hbox{in } \R^N\setminus B_1 \\
d^{1-s}u &= 0 & & \hbox{on }\partial B_1.
\end{aligned}\right.
\qquad\hbox{and}\qquad
\left\lbrace\begin{aligned}
\Ds u &= 0 & & \hbox{in }B_1 \\
u &= 0 & & \hbox{in } \R^N\setminus B_1 \\
d^{1-s}u &= 1 & & \hbox{on }\partial B_1.
\end{aligned}\right.
\]
are of the order of $O(d^{-s/2})$ and $O(d^{s-1})$ respectively at $\partial B_1$, see \cite{a}.
The existence of harmonic functions of this sort can therefore be used to prove, via a sub- and supersolution argument, 
the existence of boundary blow-up solutions to nonlinear problems
\[
\left\lbrace\begin{aligned}
\Ds u &= -f(x,u) & & \hbox{in }\Omega \\
u &= g & & \hbox{in } \R^N\setminus\Omega \\
d^{1-s}u &= h & & \hbox{on }\partial\Omega.
\end{aligned}\right.
\]
with $f(x,u)\geq 0$; anyhow this singular behaviour is driven by a linear phenomenon rather than a compensation between the nonlinearity and the explosion
(as in the classical case), indeed no growth condition on $f$ arises except when $h\not\equiv 0$,
where one needs
\[
\int_\Omega f(x,d(x)^{s-1})\,d(x)^s\;dx\ <\ \infty
\]
in order to make sense of the weak $L^1$ definition.
\smallskip

For this reason we address here the question of the existence of solutions to problems of the form
\[
\left\lbrace\begin{aligned}
\Ds u &= -f(u) & &\hbox{in }\Omega \\
u &= g & & \hbox{in } \R^N\setminus\Omega, \\
d^{1-s}u &= +\infty & & \hbox{on }\partial\Omega
\end{aligned}\right.
\quad g\geq 0,\ \ \int_{\R^N\setminus\Omega}g\,d^{-s}\min\{1,d^{-N-s}\}=+\infty.
\]
providing sufficient conditions for the solvability.
In doing so, we extend the results by Felmer and Quaas \cite{felmer-quaas} 
and Chen, Felmer and Quaas \cite{chen-felmer} for $f(u)=u^p$, which is the only bibliography available on the topic,
and we also clarify the notion of {\it large solution} in this setting.
The results listed in Theorems \ref{main1} and \ref{main2} below
can be applied for a particular case of the {\it fractional singular Yamabe problem},
see e.g. Gonz\'alez, Mazzeo and Sire \cite{mar-sire}.

\subsection{Hypotheses and main results}

We work in the following set of assumptions:
\begin{itemize}
\item $\Omega$ is a bounded open domain of class $C^2$,
\item $f$ is an increasing $C^1$ function with $f(0)=0$,
\item $F$ is the antiderivative of $f$ vanishing in 0:
\begin{equation}\label{F}
F(t)\ :=\ \int_0^tf(\tau)\;d\tau,
\end{equation}
\item there exist $0<m<M$, such that 
\begin{equation}\label{tech}
1+m\leq \frac{tf'(t)}{f(t)}\leq 1+M,
\end{equation}
and thus $f$ satisfies \eqref{ko} because, integrating the lower inequality,
one gets
\[
f(t)\geq f(1)t^{1+m}\qquad\hbox{and}\qquad
F(t)\geq\frac{f(1)}{2+m}\,t^{2+m};
\]
we can therefore define the function
\begin{equation}\label{phi}
\phi(u)\ :=\ \int_u^{+\infty}\frac{dt}{\sqrt{F(t)}},
\end{equation}
\item $f$ satisfies 
\begin{equation}\label{L1}
\int_1^{+\infty}\phi(t)^{1/s}\;dt\ <\ +\infty.
\end{equation}
\end{itemize}

In what follows we will use the expression $g\asymp h$ where $g,h:(0,+\infty)\to(0,+\infty)$ to shorten
\[
\hbox{there exists }C>0 \hbox{ such that } \frac{h(t)}{C}\leq g(t)\leq Ch(t),\hbox{ for any }t>0.
\]

\begin{rmk}\label{phi-rmk}\rm The function $\phi:(0,+\infty)\rightarrow(0,+\infty)$ is monotone decreasing and
\[
\lim_{t\downarrow0}\phi(t)=+\infty,\qquad\lim_{t\uparrow+\infty}\phi(t)=0.
\]
Moreover
\[
\phi'(u)=-\frac{1}{\sqrt{F(u)}}
\]
is of the same order as $-(u\,f(u))^{-1/2}$ since for $t>0$ and some $\tau\in(0,t)$, 
by the Cauchy Theorem,
\[
\frac{F(t)}{t\,f(t)}=\frac{f(\tau)}{f(\tau)+\tau\,f'(\tau)}\ 
\left\lbrace\begin{aligned}
& \geq\frac{1}{2+M}, \\
& \leq\frac{1}{2+m}.
\end{aligned}\right.
\]
This entails that the order of $\phi(u)$ is the same as $(u/f(u))^{1/2}$ indeed
for $u>0$ and some $t\in(u,+\infty)$
\[
\frac{\sqrt{\frac{u}{f(u)}}}{\phi(u)}=\frac{\frac{1}{2}\sqrt{\frac{f(t)}{t}}\cdot\frac{f(t)-tf'(t)}{f(t)^2}}{\phi'(t)}
\asymp \frac{f(t)-tf'(t)}{-f(t)}= \frac{tf'(t)}{f(t)}-1
\]
which belongs to $(m,M)$ by hypothesis \eqref{tech}.
Note that hypothesis \eqref{L1} is therefore equivalent to
\begin{equation}\label{L1-bis}
\int_1^{+\infty}\left(\frac{t}{f(t)}\right)^{\frac1{2s}}\;dt\ <\ +\infty.
\end{equation}
\end{rmk}

\begin{rmk}\label{8888}\rm
In \cite{keller} and \cite{osserman} condition \eqref{ko}
is proven to be necessary and sufficient for the existence of a solution of
\[
\left\lbrace\begin{aligned}
-\lapl u = -f(u) \hbox{ in }\Omega, \\
\lim_{x\to\partial\Omega}u(x) =  +\infty. 
\end{aligned}\right.
\]
Note that if we set $s=1$ in \eqref{L1} then
\[
+\infty\ >\ \int_{u}^{+\infty}\phi(t)\;dt\asymp\int_{u}^{+\infty}\sqrt\frac{t}{f(t)}\;dt
\asymp\int_{u}^{+\infty}\frac{t}{\sqrt{F(t)}}\;dt 
\]
we get the condition to force the classical solution $u$ to be $L^1(\Omega)$.
Indeed in \cite[Theorem 1.6]{ddgr} it is proved that a solution $u$ satisfies
\begin{equation}\label{ddgr}
\lim_{x\rightarrow\partial\Omega}\frac{\phi(u(x))}{d(x)}=1
\end{equation}
which yields that $u\in L^1(\Omega)$ if and only if $\phi^{-1}$,
the inverse function of $\phi$ (recall it is monotone decreasing),
is integrable in a neighbourhood of $0$, 
i.e. with a change of integration variable
\[
+\infty>\int_0^\eta\phi^{-1}(r)\;dr
=\int_{\phi^{-1}(\eta)}^{+\infty}t|\phi'(t)|\;dt
=\int_{t_0}^{+\infty}\frac{t}{\sqrt{F(t)}}\;dt.
\]
\end{rmk}

Our results can be summarized as follows.

\begin{theo}\label{main1} Suppose that the nonlinear term $f$ satisfies hypotheses \eqref{tech} and \eqref{L1} above and
\begin{equation}\label{E}
\int_{t_0}^{+\infty}f(t)t^{-2/(1-s)}\;dt\ <\ +\infty.
\end{equation}
Then problem
\begin{equation}\label{problem}
\left\lbrace\begin{aligned}
\Ds u &= -f(u) & & \hbox{ in }\Omega \\
u &= 0 & & \hbox{ in }\R^N\setminus\Omega \\
d^{1-s}u &= +\infty & & \hbox{ on }\partial\Omega
\end{aligned}\right.
\end{equation}
admits a solution $u\in L^1(\Omega)$. 
Moreover there exists $c>0$ for which
\begin{equation}\label{bbehav}
\phi(u(x))\ \geq\ c\,d(x)^s\qquad\hbox{in }\Omega.
\end{equation}
\end{theo}

\begin{rmk}\rm The condition $u\in L^1(\Omega)$ is necessary to make sense
of the fractional Laplacian, see equation \eqref{fr-lapl}. Also,
compare the boundary behaviour in this setting expressed by equation \eqref{bbehav},
with the classical one in equation \eqref{ddgr}.
\end{rmk}

\begin{theo}\label{main2} Suppose that the nonlinear term $f$ satisfies hypotheses \eqref{tech} and \eqref{L1} above and 
\begin{eqnarray}
& g:\R^N\setminus\Omega\longrightarrow[0,+\infty),\qquad g\in L^1(\R^N\setminus\Omega) & \nonumber \\
& \phi(g(x))\ \geq\ d(x)^s,\qquad\hbox{near }\partial\Omega. & \label{g2}
\end{eqnarray}
Then problem
\begin{equation}\label{gproblem}
\left\lbrace\begin{aligned}
\Ds u &= -f(u) & \hbox{ in }\Omega \\
u &= g & \hbox{ in }\R^N\setminus\Omega 
\end{aligned}\right.
\end{equation}
admits a solution $u\in L^1(\Omega)$. 
Moreover there exists $c>0$ for which
\[
\phi(u(x))\ \geq\ c\,d(x)^s\qquad\hbox{near }\partial\Omega.
\]
\end{theo}

\begin{rmk}\rm 
Mind that in problem \eqref{gproblem} we do not prescribe the singular trace at $\partial\Omega$.
\end{rmk}

\begin{rmk}\rm
The hypotheses in Theorem \ref{main1}, when considering $f(u)=u^p$, reduce to
\[
p\in\left(1+2s,1+\frac{2s}{1-s}\right),
\]
see Theorem \ref{power-th}.
Note that this range of exponents does {\it not}
converge, letting $s\uparrow 1$, to the set of admissible exponents 
for $-\lapl$, which is given by \eqref{ko} and simply reads as $p\in(1,+\infty)$.
Indeed, we only have $1+2s\to 3$ as $s\uparrow1$.
This is not discouraging though.
In this fractional setting we need $u\in L^1(\Omega)$
to make sense of the operator: this is an additional (natural) restriction we do not have
in the classical problem, so it is reasonable to get smaller 
ranges for $p$. Moreover, the classical solution to the large problem
is known to behave like (cf. equation \eqref{ddgr})
\[
u\ \asymp\ d^{-2/(p-1)}
\]
and such a $u$ is in $L^1(\Omega)$ when $p>3$.
In this sense we actually have the asymptotic convergence 
of the admissible ranges of exponents. Compare this also 
with Remark \ref{8888}.
\end{rmk}

\begin{rmk}\rm
As it will be clear in the following proofs,
hypothesis \eqref{tech} is technical and not structural.
We conjecture that it is not necessary to
establish existence results.
But let us mention how a similar assumption
arises naturally even in the classical framework
when dealing with the computation of the 
asymptotic behaviour of the solution:
see Bandle and Marcus \cite[equations (B) and (B)$'$]{bandle-marcus}.
\end{rmk}

The strategy to prove the existence result in Theorem \ref{main1}
is to build the sequence ${\{u_k\}}_{k\in\N}$ of solutions to\footnote{The operator
$E$ denotes the singular trace operator defined in \cite{a}.}
\begin{equation}\label{approx}
\left\lbrace\begin{aligned}
\Ds u_k &= -f(u_k) & & \hbox{ in }\Omega \\
u_k &= 0 & & \hbox{ in }\R^N\setminus\Omega \\
Eu_k &= k & & \hbox{ on }\partial\Omega
\end{aligned}\right.
\end{equation}
and then let $k\uparrow+\infty$. In case $u_k$ admits 
a limit, then we will need to prove that this is the 
solution we were looking for. This might also be called 
the {\it minimal large solution} by borrowing the expression 
used in the classical theory.

We can also provide a partial nonexistence result.
\begin{theo}\label{nonexist} Suppose there exist $a,b>0$ for which
\begin{equation}
f(t)\leq a+bt,\qquad \hbox{for any }t\in(0,+\infty).
\end{equation}
Then there exists $\alpha>0$
\[
u_k(x)\uparrow+\infty\quad\hbox{ as }k\uparrow+\infty,\qquad
\hbox{ whenever }d(x)<\alpha.
\]
\end{theo}

In the case of power-like nonlinearities we can show the following.
\begin{theo}\label{power-th} Let $f(t)=t^p,\ p>0$. Then
\begin{enumerate}
\item if $p\in\left[1+\frac{2s}{1-s},+\infty\right)$ then
the approximating sequence ${\{u_k\}}_{k\in\N}$ does not exist;
\item if $p\in\left(1+2s,1+\frac{2s}{1-s}\right)$ then the approximating
sequence converges to a solution $u$ of \eqref{problem} and
\[
u\ \leq\ C\,\delta^{-2s/(p-1)};
\]
\item if $p\in(1,1+s)$ then the approximating sequence exits $L^1(\Omega)$,
meaning $\|u_k\|_{L^1(\Omega)}\uparrow+\infty$ as $k\uparrow+\infty$.
\item if $p\in(0,1]$ then the approximating sequence blows-up uniformly
in some open strip near the boundary.
\end{enumerate}
\end{theo}

%

\subsection{Notations}

In the following we will always denote by $\C E=\R^N\setminus E$
for any $E\subset\R^N$.

Hypothesis \eqref{tech} implies that $f(t)t^{-1-M}$ is monotone decreasing
and $f(t)t^{-1-m}$ is monotone increasing, since
\[
\frac{d}{dt}\frac{f(t)}{t^{1+M}}=\frac{1}{t^{1+M}}\left(f'(t)-(1+M)\frac{f(t)}{t}\right)\ \leq\ 0, 
\qquad
\frac{d}{dt}\frac{f(t)}{t^{1+m}}=\frac{1}{t^{1+m}}\left(f'(t)-(1+m)\frac{f(t)}{t}\right)\ \geq\ 0:
\]
we write this monotonicity conditions as
\begin{equation}\label{mono-f}
c^{1+m}f(t)\ \leq\ f(ct)\ \leq\ c^{1+M}f(t),\qquad c>1,\ t>0.
\end{equation}
The function $F$ satisfies two inequalities similar to \eqref{tech}:
\begin{equation}\label{Ftech}
2+m\ \leq\ \frac{t\,f(t)}{F(t)}\ \leq\ 2+M,
\end{equation}
indeed by integrating \eqref{tech} we deduce 
\[
(1+m)F(t)\leq\int_0^t\tau\,f'(\tau)\;d\tau=tf(t)-F(t).
\]
Let $\psi=\phi^{-1}$ be the inverse of $\phi$, so that
\begin{equation}\label{psi}
v\ =\ \int_{\psi(v)}^{+\infty}\frac{dt}{\sqrt{F(t)}},\qquad v\geq 0.
\end{equation}
The function $\psi$ is decreasing and $\psi(v)\uparrow+\infty$ as $v\downarrow 0$.
Moreover, by Remark \ref{phi-rmk} and \eqref{Ftech}, for $u>0$ and some $y\in(u,+\infty)$
\begin{equation*}
\frac{\phi(u)}{u|\phi'(u)|}=\frac{\sqrt{F(u)}}{u}\int_u^{+\infty}\frac{dt}{\sqrt{F(t)}}=
\frac{-\frac{1}{\sqrt{F(y)}}}{\frac{1}{\sqrt{F(y)}}-\frac{yf(y)}{2F(y)^{3/2}}}=
\frac{1}{\frac{yf(y)}{2F(y)}-1}\quad
\left\lbrace\begin{aligned}
& \geq \frac2M \\
& \leq \frac2m
\end{aligned}\right.
\end{equation*}
which in turn says that it holds, by setting $v=\phi(u)$,
\begin{equation}\label{psitech}
\frac2M\ \leq\ \frac{v|\psi'(v)|}{\psi(v)}\ \leq\ \frac2m,\qquad
\end{equation}
and one can prove also 
\begin{equation}\label{mono-psi}
\psi(cv)\ \leq\ c^{-2/M}\psi(v),\qquad c\in(0,1),\ v>0.
\end{equation}
as we have done for \eqref{mono-f} above. Also, by \eqref{Ftech} and \eqref{psitech},
\begin{equation}\label{mono-psi2}
\frac{v^2\,\psi''(v)}{\psi(v)}=\frac{v^2\,f(\psi(v))}{2\,\psi(v)}\asymp
\frac{v^2\,F(\psi(v))}{\psi(v)^2}=\frac{v^2\,\psi'(v)^2}{\psi(v)^2}\asymp 1.
\end{equation}

\subsection{Construction of a supersolution}

In this paragraph we prove the key point for the proof of Theorems 
\ref{main1} and \ref{main2}, that is we build a supersolution
to both problems by handling the function $U$ defined in \eqref{U} below.

Since by assumption $\partial\Omega\in C^2$, the function $\dist(x,\partial\Omega)$
is $C^2$ in an open strip around the boundary, except on $\partial\Omega$ itself.
Consider a positive function $\delta(x)$ which is obtained by
extending $\dist(x,\partial\Omega)$ smoothly to $\R^N\setminus\partial\Omega$.
Define 
\begin{equation}\label{U}
U(x)\ =\ \psi(\delta(x)^s),\qquad x\in\R^N.
\end{equation}

\begin{lem}\label{U-L1} 
The function $U$ defined in \eqref{U} is in $L^1(\Omega)$.
\end{lem}
\begin{proof}
Since both $\psi$ and $\delta^s$ are continuous in $\Omega$, then $U\in L^1_{loc}(\Omega)$.
Fix $\delta_0>0$ small and consider $\Omega_0=\{x\in\Omega:\delta(x)<\delta_0\}$. We have
(using once the coarea formula)
\[
\int_{\Omega_0}\psi(\delta(x)^s)\;dx\ 
\leq\ C\int_0^{\delta_0}\psi(t^s)\;dt:
\]
apply now the transformation $\psi(t^s)=\eta$ to get
\[
\int_{\Omega_0}U(x)\;dx \ \leq \ C\int_{\eta_0}^{+\infty}\eta\,\phi(\eta)^{(1-s)/s}|\phi'(\eta)|\;d\eta 
\]
where, by Remark \ref{phi-rmk},
\[
|\phi'(\eta)|\asymp\frac1{\sqrt{\eta\,f(\eta)}}\qquad\hbox{and}\qquad\phi(\eta)\asymp\sqrt{\frac{\eta}{f(\eta)}},
\]
therefore
\[
\int_{\Omega_0}U(x)\;dx \ \leq \ C\int_{\eta_0}^{+\infty}\left(\frac{\eta}{f(\eta)}\right)^\frac1{2s}d\eta 
\]
which is finite by \eqref{L1-bis}.
\end{proof}

The following Proposition shows that $U$ is a good starting point 
to build a supersolution. The proof is technical but this is the key step
for the following.

\begin{prop}\label{impo}
The function $U$ defined in \eqref{U} satisfies for some $C,\,\delta_0>0$
\begin{equation}\label{impo-eq}
\Ds U \geq-Cf(U),\qquad \hbox{in }\Omega_{\delta_0}=\{x\in\Omega:\delta(x)<\delta_0\}.
\end{equation}
\end{prop}
Before giving the proof, we prove a preliminary lemma.
\begin{lem} Let $\Omega\subset\R^N$ a bounded open domain
with compact boundary $\partial\Omega$. Cover $\partial\Omega$
by a finite number of open portions $\Gamma_j\subset\partial\Omega$,
$j=1,\ldots,n$. For any $\eta\in\partial\Omega$ there is some
$i(\eta)\in\{1,\ldots,n\}$ such that $\eta\in\Gamma_{i(\eta)}$ for which
\begin{equation}\label{00}
\dist(\eta,\partial\Omega\setminus\Gamma_{i(\eta)})\geq c
\end{equation}
for some constant $c>0$ independent of $\eta\in\partial\Omega$.
\end{lem}
\begin{proof}
For any $j=1,\ldots,n$, the function $\eta\mapsto\dist(\eta,\partial\Omega\setminus\Gamma_j)$ is continuous in $\partial\Omega$
and so is $\eta\mapsto\max_j\dist(\eta,\partial\Omega\setminus\Gamma_j)$: there is a point $\eta_*\in\partial\Omega$
where $\eta\mapsto\max_j\dist(\eta,\partial\Omega\setminus\Gamma_j)$ attains its minimum. 
Such a minimum cannot be $0$ because $\eta_*$ belongs at least to one of the $\Gamma_j$.
This implies that for any $\eta\in\partial\Omega$ there exists
$i(\eta)\in\{1,\ldots,n\}$ such that
\[
\max_j\dist(\eta_*,\partial\Omega\setminus\Gamma_j)
\leq \max_j\dist(\eta,\partial\Omega\setminus\Gamma_j)=
\dist(\eta,\partial\Omega\setminus\Gamma_{i(\eta)}).
\]
\end{proof}

\begin{proof}{\it of Proposition \ref{impo}.}
We start by writing, for $x\in\Omega$
\begin{equation}\label{split}
\frac{\Ds U(x)}{\A(N,s)}=PV\int_\Omega\frac{\psi(\delta(x)^s)-\psi(\delta(y)^s)}{|x-y|^{N+2s}}\:dy
+\int_{\C\Omega}\frac{\psi(\delta(x)^s)-\psi(\delta(y)^s)}{|x-y|^{N+2s}}\:dy.
\end{equation}
Let us begin with an estimate for 
\[
PV\int_\Omega\frac{\psi(\delta(x)^s)-\psi(\delta(y)^s)}{|x-y|^{N+2s}}\:dy.
\]
Split the integral into
\[
\int_{\Omega_1}\frac{\psi(\delta(x)^s)-\psi(\delta(y)^s)}{|x-y|^{N+2s}}\:dy
+PV\int_{\Omega_2}\frac{\psi(\delta(x)^s)-\psi(\delta(y)^s)}{|x-y|^{N+2s}}\:dy+
\int_{\Omega_3}\frac{\psi(\delta(x)^s)-\psi(\delta(y)^s)}{|x-y|^{N+2s}}\:dy
\]
where we have set
\begin{align*}
\Omega=\Omega_1\cup\Omega_2\cup\Omega_3,\text{ with: } 
 & \Omega_1=\left\{y\in\Omega:\delta(y)>\frac32\delta(x)\right\} \\
 & \Omega_2=\left\{y\in\Omega:\frac12\delta(x)\leq\delta(y)\leq\frac32\delta(x)\right\} \\
 & \Omega_3=\left\{y\in\Omega:\delta(y)<\frac12\delta(x)\right\}.
\end{align*}
In $\Omega_1$ we have in particular $\delta(y)>\delta(x)$ so that,
since $\psi$ decreasing function, the first integral contributes by a positive quantity.
Now, let us turn to integrals on $\Omega_2$ and $\Omega_3$.
Set $x=\theta+\delta(x)\grad\delta(x),\ \theta\in\partial\Omega$:
up to a rotation and a translation,
we can suppose that $\theta=0$ and $\grad\delta(x)=e_N$.

Let ${\{\Gamma_j\}}_{j=1}^n$ be a finite open covering of $\partial\Omega$ and 
let $\Gamma:=\Gamma_{i(0)}$ (in the notations of the last lemma)
be a neighbourhood of $0$ on $\partial\Omega$ chosen from $\{\Gamma_j\}_{j=1}^n$ and
for which \eqref{00} is fulfilled. Let also
\[
\omega=\{y\in\R^N:y=\eta+\delta(y)\grad\delta(y),\ \eta\in\Gamma\}. 
\]
The set $\Gamma\subset\partial\Omega$ can be described via as the graph of a $C^2$ function
\begin{eqnarray*}
\gamma:B'_r(0)\subseteq\R^{N-1} & \longrightarrow & \R \\
\eta' & \longmapsto & \gamma(\eta')\quad\text{ s.t }\eta=(\eta',\gamma(\eta'))\in\Gamma
\end{eqnarray*}
satisfying $\gamma(0)=|\grad\gamma(0)|=0$.

The integration on $(\Omega_2\cup\Omega_3)\setminus\omega$
is lower order with respect to the one on $(\Omega_2\cup\Omega_3)\cap\omega$
since in the latter we have the singularity in $x$ to deal with,
while in the former $|x-y|$ is a quantity bounded below independently on $x$.
Indeed when $y\in(\Omega_2\cup\Omega_3)\setminus\omega$
\[
|x-y|\geq|\eta+\delta(y)\grad\delta(y)|-\delta(x)\geq|\eta|-\delta(y)-\delta(x)\geq\dist(0,\partial\Omega\setminus\Gamma)-\frac{5}{2}\delta(x)
\]
where $\delta(x)$ is small and the first addend is bounded uniformly in $x$ by \eqref{00}.

We are left with:
\[
C\cdot PV\int_{\Omega_2\cap\omega}\frac{\psi(\delta(x)^s)-\psi(\delta(y)^s)}{|x-y|^{N+2s}}\:dy+
C\int_{\Omega_3\cap\omega}\frac{\psi(\delta(x)^s)-\psi(\delta(y)^s)}{|x-y|^{N+2s}}\:dy.
\]
Let us split the remainder of the estimate in steps.

{\it First step: the distance between $x$ and $y$.} We claim that there exists $c>0$ such that
\begin{equation}\label{distxy}
\begin{aligned}
& |x-y|^2\ \geq\ c\left(|\delta(x)-\delta(y)|^2+|\eta'|^2\right),\quad y\in(\Omega_2\cup\Omega_3)\cap\omega,\\ 
& y=\eta+\delta(y)\grad\delta(y),\ \eta=(\eta',\gamma(\eta')).
\end{aligned}
\end{equation}
Since in our set of coordinates $x=\delta(x)e_N$, we can write
\begin{multline*}
|x-y|^2=|\delta(x)e_N-\delta(y)e_N+\delta(y)e_N-y_Ne_N-y'|^2\ \geq \\
\geq\ |\delta(x)-\delta(y)|^2-2|\delta(x)-\delta(y)|\cdot|\delta(y)-y_N|+|\delta(y)-y_N|^2+|y'|^2.
\end{multline*}
We concentrate our attention on $|\delta(y)-y_N|$: the idea is to show
that this is a small quantity; indeed, in the particular case when $\Gamma$ 
lies on the hyperplane $y_N=0$, this quantity is actually zero. As in the definition
of $\omega$, we let $y=\eta+\delta(y)\grad\delta(y)$ and $\eta=(\eta',\gamma(\eta'))\in\Gamma$:
thus $y_N=\gamma(\eta')+\delta(y)\langle\grad\delta(y),e_N\rangle$ where 
$\grad\delta(y)$ is the inward unit normal to $\partial\Omega$ at the point $\eta$, so that
\[
\grad\delta(y)=\frac{(-\grad\gamma(\eta'),1)}{\sqrt{|\grad\gamma(\eta')|^2+1}}
\]
and
\begin{equation}\label{yeta}
y'=\eta'-\frac{\delta(y)\,\grad\gamma(\eta')}{\sqrt{|\grad\gamma(\eta')|^2+1}},
\qquad y_N=\gamma(\eta')+\frac{\delta(y)}{\sqrt{|\grad\gamma(\eta')|^2+1}}.
\end{equation}
Now, since $y\in\Omega_2\cup\Omega_3$, it holds $|\delta(x)-\delta(y)|\leq\delta(x)$ and
\[
|\delta(y)-y_N|\leq |\gamma(\eta')|+\delta(y)\left(1-\frac{1}{\sqrt{|\grad\gamma(\eta')|^2+1}}\right)
\leq C|\eta'|^2+2C\delta(x)|\eta'|^2
\]
where, in this case, $C=\|\gamma\|_{C^2(B_r)}$ depends only on the geometry of $\partial\Omega$ and not on $x$.
By \eqref{yeta}, we have 
\[
|\eta'|^2\leq 2|y'|^2+\frac{2\delta(y)^2\,|\grad\gamma(\eta')|^2}{|\grad\gamma(\eta')|^2+1}
\leq 2|y'|^2+2C\delta(y)^2\,|\eta'|^2\leq 2|y'|^2+C\delta(x)^2\,|\eta'|^2,
\]
so that $|\eta'|^2\leq C|y'|^2$ when $\delta(x)$ is small enough.
Finally
\begin{multline*}
|x-y|^2\geq |\delta(x)-\delta(y)|^2+|y'|^2-2|\delta(x)-\delta(y)|\cdot|\delta(y)-y_N|\ \geq \\
\geq\ |\delta(x)-\delta(y)|^2+c|\eta'|^2-2C\delta(x)|\eta'|^2,
\end{multline*}
where, again, $C=\|\gamma\|_{C^2(B_r)}$ and \eqref{distxy} is proved provided $x$ is close enough to $\partial\Omega$.

{\it Second step: integration on $\Omega_2\cap\omega$.} Using the regularity of $\psi$ and $\delta$ we write
\[
\psi(\delta(x)^s)-\psi(\delta(y)^s)\geq \grad(\psi\circ\delta^s)(x)\cdot(x-y)
-\|D^2(\psi\circ\delta^s)\|_{L^\infty(\Omega_2\cap\omega)}|x-y|^2
\]
where
\[
D^2(\psi\circ\delta^s)=\frac{s\psi'(\delta^s)}{\delta^{1-s}}\,D^2\delta
+\frac{s^2\,\psi''(\delta^s)}{\delta^{2-2s}}\,\grad\delta\otimes\grad\delta
+\frac{s(s-1)\,\psi'(\delta^s)}{\delta^{2-s}}\,\grad\delta\otimes\grad\delta
\]
so that 
\[
\|D^2(\psi\circ\delta^s)\|_{L^\infty(\Omega_2\cap\omega)}\ \leq\ 
C\left\|\frac{\psi'(\delta^s)}{\delta^{1-s}}\right\|_{L^\infty(\Omega_2\cap\omega)}
+C\left\|\frac{\psi''(\delta^s)}{\delta^{2-2s}}\right\|_{L^\infty(\Omega_2\cap\omega)}
+C\left\|\frac{\psi'(\delta^s)}{\delta^{2-s}}\right\|_{L^\infty(\Omega_2\cap\omega)}.
\]
By definition of $\Omega_2$ and by \eqref{mono-psi} we can control the sup-norm by 
the value at $x$:
\begin{multline*}
\|D^2(\psi\circ\delta^s)\|_{L^\infty(\Omega_2\cap\omega)}\leq
C\,\frac{|\psi'(\delta(x)^s)|}{\delta(x)^{1-s}}
+C\,\frac{\psi''(\delta(x)^s)}{\delta(x)^{2-2s}}
+C\,\frac{|\psi'(\delta(x)^s)|}{\delta(x)^{2-s}}\ \leq \\
\leq\ C\,\frac{\psi''(\delta(x)^s)}{\delta(x)^{2-2s}}
+C\,\frac{|\psi'(\delta(x)^s)|}{\delta(x)^{2-s}}
\end{multline*}
and using equation \eqref{mono-psi2} we finally get 
\[
\|D^2(\psi\circ\delta^s)\|_{L^\infty(\Omega_2\cap\omega)}\leq 
C\,\frac{\psi''(\delta(x)^s)}{\delta(x)^{2-2s}}.
\]
If we now retrieve the whole integral and exploit \eqref{distxy}
\begin{multline*}
PV\int_{\Omega_2\cap\omega}\frac{\psi(\delta(x)^s)-\psi(\delta(y)^s)}{|x-y|^{N+2s}}\:dy\ \geq\ 
-C \frac{\psi''(\delta(x)^s)}{\delta(x)^{2-2s}}
\int_{\Omega_2\cap\omega}\frac{dy}{|x-y|^{N+2s-2}}\ \geq\\
\geq\ -C \frac{\psi''(\delta(x)^s)}{\delta(x)^{2-2s}}
\int_{\Omega_2\cap\omega}\frac{dy}{\left(|\delta(x)-\delta(y)|^2+|\eta|^2\right)^{(N+2s-2)/2}}.
\end{multline*}
We focus our attention on the integral on the right-hand side: by the coarea formula
\begin{align*}
& \int_{\Omega_2\cap\omega}\frac{dy}{\left(|\delta(x)-\delta(y)|^2+|\eta|^2\right)^{(N+2s-2)/2}}\ = \\
& =\ \int_{\delta(x)/2}^{3\delta(x)/2}dt \int_{\{\delta(y)=t\}\cap\omega}
\frac{d\sigma(\eta)}{\left(|\delta(x)-t|^2+|\eta|^2\right)^{(N+2s-2)/2}}\\
& \leq\ C\int_{\delta(x)/2}^{3\delta(x)/2}dt \int_{B_r}\frac{d\eta'}{\left(|\delta(x)-t|^2+|\eta'|^2\right)^{(N+2s-2)/2}} \\
& \leq\ C\int_{\delta(x)/2}^{3\delta(x)/2}dt \int_0^r\frac{\rho^{N-2}}{\left(|\delta(x)-t|^2+\rho^2\right)^{(N+2s-2)/2}}\;d\rho \\
& \leq\ C\int_{\delta(x)/2}^{3\delta(x)/2}dt \int_0^r\frac{\rho}{\left(|\delta(x)-t|^2+\rho^2\right)^{(2s+1)/2}}d\rho\ \leq  
\ C\int_{\delta(x)/2}^{3\delta(x)/2}\frac{dt}{|t-\delta(x)|^{2s-1}}.
\end{align*} 
We can retrieve now the chain of inequalities we stopped above:
\[
\int_{\Omega_3\cap\omega}\frac{\psi(\delta(x)^s)-\psi(\delta(y)^s)}{|x-y|^{N+2s}}\:dy\ \geq\ 
-C\frac{\psi''(\delta(x)^s)}{\delta(x)^{2-2s}}
\int_{\delta(x)/2}^{3\delta(x)/2}\frac{dt}{|\delta(x)-t|^{-1+2s}}
\geq\ -C\,\psi''(\delta(x)^s).
\]

{\it Third step: integration on $\Omega_3\cap\omega$.}
We use \eqref{distxy} once again:
\begin{align*}
& \int_{\Omega_3\cap\omega}\frac{\psi(\delta(x)^s)-\psi(\delta(y)^s)}{|x-y|^{N+2s}}\:dy\ \geq\\
& \geq\ -\int_{\Omega_3\cap\omega}\frac{\psi(\delta(y)^s)}{|x-y|^{N+2s}}\:dy\ 
\geq\ -C\int_{\Omega_3\cap\omega}\frac{\psi(\delta(y)^s)}{\left(|\delta(x)-\delta(y)|^2+|\eta'|^2\right)^\frac{N+2s}2}\:dy\\
& \geq\ -C\int_0^{\delta(x)/2}\frac{\psi(t^s)}{(\delta(x)-t)^{1+2s}}\;dt\ \geq
\ -\frac{C}{\delta(x)^{1+2s}}\int_0^{\delta(x)/2}\psi(t^s)\;dt.
\end{align*}
The term we have obtained is of the same order of $\delta(x)^{-2s}\psi(\delta(x)^s)$, by \eqref{psitech}:
\[
\int_0^{\delta(x)/2}\psi(t^s)\;dt\asymp\int_0^{\delta(x)/2}t^s\psi'(t^s)\;dt=
\frac{\delta(x)}{2s}\psi\left(\frac{\delta(x)^s}{2^s}\right)-\frac1s\int_0^{\delta(x)/2}\psi(t^s)\;dt
\]
so that
\begin{equation}\label{03}
\int_0^{\delta(x)/2}\psi(t^s)\;dt\asymp\delta(x)\psi(\delta(x)^s)=\delta(x)^{1+2s}\cdot\frac{\psi(\delta(x)^s)}{\delta(x)^{2s}}.
\end{equation}
Recall now that $\psi(\delta(x)^s)\delta(x)^{-2s}$ is in turn of the same size of $\psi''(\delta(x)^s)$ by \eqref{mono-psi2}.

{\it Fourth step: the outside integral in \eqref{split}.} 
We focus now our attention on
\[
\int_{\C\Omega}\frac{\psi(\delta(y)^s)-\psi(\delta(x)^s)}{|x-y|^{N+2s}}\;dy.
\]
First, by using the monotonicity of $\psi$, we write
\begin{multline*}
\int_{\C\Omega}\frac{\psi(\delta(y)^s)-\psi(\delta(x)^s)}{|x-y|^{N+2s}}\;dy\ \leq\ 
\int_{\{y\in\C\Omega:\delta(y)<\delta(x)\}\cap\omega}\frac{\psi(\delta(y)^s)-\psi(\delta(x)^s)}{|x-y|^{N+2s}}\;dy\ +\\
+\ \int_{\{y\in\C\Omega:\delta(y)<\delta(x)\}\setminus\omega}\frac{\psi(\delta(y)^s)-\psi(\delta(x)^s)}{|x-y|^{N+2s}}\;dy
\end{multline*}
The second integral gives
\[
\int_{\{y\in\C\Omega:\delta(y)<\delta(x)\}\setminus\omega}\frac{\psi(\delta(y)^s)-\psi(\delta(x)^s)}{|x-y|^{N+2s}}\;dy
\ \leq\ C\|\psi(\delta^s)\|_{L^1(\R^N)}
\]
because the distance between $x$ and $y$ is bounded there.
Again we point out that
\begin{align*}
& \int_{\{\delta(y)<\delta(x)\}\cap\omega}\frac{\psi(\delta(y)^s)-\psi(\delta(x)^s)}{|x-y|^{N+2s}}\;dy
\leq C\int_0^{\delta(x)}\frac{\psi(t^s)-\psi(\delta(x)^s)}{|\delta(x)+t|^{1+2s}}\;dt\ \leq \\
& \leq\ C\int_0^{\delta(x)/2}\frac{\psi(t^s)}{|\delta(x)+t|^{1+2s}}\;dt
+C\int_{\delta(x)/2}^{\delta(x)}\frac{\psi(t^s)}{|\delta(x)+t|^{1+2s}}\;dt \\
& \leq\ C\delta(x)^{-1-2s}\int_0^{\delta(x)/2}\psi(t^s)\;dt
+C\psi\left(\frac{\delta(x)^s}{2^s}\right)\int_{\delta(x)/2}^{\delta(x)}(\delta(x)+t)^{-1-2s}\;dt \\
& \leq\ C\delta(x)^{-1-2s}\int_0^{\delta(x)}\psi(t^s)\;dt
+C\psi(\delta(x)^s)\delta(x)^{-2s}
\end{align*}
which is of the order of $\psi''(\delta(x)^s)$, by \eqref{03} and \eqref{mono-psi2}.
\smallskip

{\it Conclusion.} We have proved that for $\delta(x)$ sufficiently small
\[
\Ds U(x)\geq -C\psi''(\delta(x)^s).
\]
Recall now that $\psi''(\delta^s)=f(\psi\circ\delta^s)$ and $U=\psi\circ\delta^s$ in $\Omega$,
so that
\[
\Ds U\ \geq\ -Cf(U)
\]
holds in a neighbourhood of $\partial\Omega$.
\end{proof}

Starting from $U$, it is possible to build
a full supersolution in view of the following lemma.
\begin{lem}\label{mod-supersol} Let $v:\R^N\to\R$ a function which satisfies
$\Ds v\in C(\Omega)$.
If there exist $C,\,\delta_0>0$ such that
\[
\Ds v \ \geq\ -Cf(v) \qquad \hbox{ in }\Omega_{\delta_0}:=\{x\in\Omega:\delta(x)<\delta_0\}
\]
then there exists $\overline{u}\geq v$ such that $\Ds\overline{u}\geq-f(\overline{u})$ throughout $\Omega$. 
\end{lem}
\begin{proof}
Define $\xi:\R^N\to\R$ as the solution to
\begin{equation}\label{xi}
\left\lbrace\begin{aligned}
\Ds\xi &=1 & \hbox{ in }\Omega \\
\xi &=0 & \hbox{ in }\C\Omega \\
E\xi &=0 & \hbox{ on }\partial\Omega
\end{aligned}\right.
\end{equation}
and consider $\overline u=\mu v+\lambda\xi$, where $\mu,\lambda\geq1$. If $C\in(0,1]$ then 
$\Ds v\geq -f(v)$ in $\Omega_{\delta_0}$, so choose $\mu=1$. If $C>1$, then choose $\mu=C^{1/M}>1$
in order to have in $\Omega_{\delta_0}$
\[
\Ds\overline u+f(\overline u)=\Ds (\mu v+\lambda\xi)+f(\mu v+\lambda\xi)\geq -\mu\,Cf(v)+f(\mu v)\ \geq \ (-\mu\,C+\mu^{1+M})f(v)=0
\]
where we have heavily used the positivity of $\xi$ and \eqref{mono-f}.
Now, since $\Ds v\in C(\overline{\Omega\setminus\Omega_{\delta_0}})$ we can choose $\lambda=\mu\|\Ds v\|_{L^\infty(\Omega\setminus\Omega_{\delta_0})}$
so that also in $\Omega\setminus\Omega_{\delta_0}$
\[
\Ds \overline u=\Ds(\mu v+\lambda\xi)=\mu\Ds v+\lambda\geq 0\geq -f(\overline u).
\]
\end{proof}

\section{existence}\label{exist}

\begin{lem}\label{EU} If the nonlinear term $f$ satisfies the growth condition \eqref{E}
then the function $U$ defined in \eqref{U} satisfies
\[
\lim_{x\to\partial\Omega}\delta(x)^{1-s}U(x) =\ +\infty.
\]
\end{lem}
\begin{proof} Write
\[
\liminf_{x\to\partial\Omega}\delta(x)^{1-s}\psi(\delta(x)^{s})\ =
\ \liminf_{u\uparrow+\infty}u\,\phi(u)^{\frac{1-s}{s}}.
\]
Such a limit is $+\infty$ if and only if 
\[
\liminf_{u\uparrow+\infty}u^{\frac{s}{1-s}}\int_u^{+\infty}\frac{dt}{\sqrt{2F(t)}}=+\infty.
\]
If we use the L'H\^opital's rule to
\[
\frac{\displaystyle\int_u^{+\infty}\dfrac{dt}{\sqrt{2F(t)}}}{u^{-s/(1-s)}}
\] 
we get the ratio $u^{\frac{1}{1-s}}/\sqrt{2F(u)}$
and applying once again the L'H\^opital's rule, this time to
$u^{\frac{2}{1-s}}/F(u)$,
we get $u^{\frac{1+s}{1-s}}/f(u)$
which diverges by hypothesis \eqref{E}. Indeed, since $f$ is increasing,
\[
u^{-\frac{1+s}{1-s}}f(u)= f(u)\cdot\frac{1-s}{1+s}\int_u^{+\infty}t^{-2/(1-s)}dt\leq \int_u^{+\infty}f(t)t^{-2/(1-s)}dt
\xrightarrow[u\uparrow+\infty]{}0.
\] 
\end{proof}

Collecting the information so far, 
we have that Lemmata \ref{U-L1}, \ref{mod-supersol} and \ref{EU} 
fully prove the following theorems.
\begin{theo}\label{U-super}
If the nonlinear term $f$ satisfies the growth condition \eqref{EU},
then there exists a function $\overline{u}$ supersolution to \eqref{problem}.
Moreover
\[
\overline{u}\ =\ \mu \psi(\delta^s)+\lambda\xi,\qquad\hbox{in }\Omega
\]
where $\xi$ is the solution of \eqref{xi}, $\lambda>0$, $\mu=\max\{1,C^{1/M}\}$ where $C>0$ is the constant in
\eqref{impo-eq} and $M>0$ the one in \eqref{tech}.
\end{theo}

\begin{theo}\label{U-gsuper}
There exists a function $\overline{u}$ supersolution to \eqref{gproblem}.
Moreover
\[
\overline{u}\ =\ \mu \psi(\delta^s)+\lambda\xi,\qquad\hbox{in }\Omega
\]
where $\xi$ is the solution of \eqref{xi}, $\lambda>0$, $\mu=\max\{1,C^{1/M}\}$ where $C>0$ is the constant in
\eqref{impo-eq} and $M>0$ the one in \eqref{tech}.
\end{theo}

\subsection{Proof of Theorem \ref{main1}}\label{proof-main}

Theorems \ref{U-super} bears as a consequence the following.
Build the sequence of solutions to problems
\begin{equation}\label{kproblem}
\left\lbrace\begin{aligned}
\Ds u_k\  &=\ -f(u_k) & & \hbox{ in }\Omega \\
u_k\ &=\ 0 & & \hbox{ in }\C\Omega \\
E\,u_k\ &=\ k & & \hbox{ on }\partial\Omega, & k\in\N.
\end{aligned}\right.
\end{equation}
The existence of any $u_k$ can be proved as in \cite[Theorem 1.2.12]{a}, in view of hypothesis \eqref{E},
since it implies
\[
\int_0^{\delta_0}f(\delta^{s-1})\delta^s\;d\delta\ <\ +\infty.
\]

The first tool we need is a Comparison Principle.

\begin{lem}[comparison principle]\label{comprinc} 
Let $v,w\in C(\Omega)\cap L^1(\Omega)$ solve pointwisely 
\[
\left\lbrace\begin{aligned}
\Ds v\  &\leq\ -f(v) & & \hbox{ in }\Omega \\
v\ &\leq\ 0 & & \hbox{ in }\C\Omega 
\end{aligned}\right.
\qquad\hbox{and}\qquad
\left\lbrace\begin{aligned}
\Ds w\  &\geq\ -f(w) & & \hbox{ in }\Omega \\
w\ &\geq\ 0 & & \hbox{ in }\C\Omega.
\end{aligned}\right.
\]
If $v\leq w$ in $U_{\alpha}:=\{x\in\Omega:\delta(x)<\alpha\}$ for some $\alpha>0$, then $v\leq w$ in the whole $\Omega$.
\end{lem}
\begin{proof}
Consider $\Omega^+=\{v>w\}\subset(\Omega\setminus U_\alpha)$. The difference $v-w$ achieves its (global) maximum
at some point $x^*\in\Omega^+$. So
\[
0<\Ds(v-w)(x^*)\leq f(w(x^*))-f(v(x^*))\leq 0
\]
in view of the monotonicity of $f$. Thus $\Omega^+$ must be empty.
\end{proof}

\paragraph*{\it Step 1: $\{u_k\}_{k\in\N}$ has a pointwise limit.}
Any $u_k$ solves the equation in a pointwise sense, as Lemma \ref{weak-point} below implies.
The sequence $\{u_k\}_{k\in\N}$ is increasing with $k$ by the Comparison Principle, Lemma \ref{comprinc}.
Moreover any $u_k$ lies below $\overline{u}$: since $E(\overline{u}-u_k)=+\infty$,
then $u_k\leq\overline{u}$ holds close to $\partial\Omega$
and another application of the Comparison Principle yields $u_k\leq\overline{u}$ in $\Omega$.

Finally, $\{u_k\}_{k\in\N}$ is increasing and pointwisely bounded by $\overline{u}$ throughout $\Omega$.
This entails that
\[
u(x)\ :=\ \lim_{k\uparrow+\infty}u_k(x)
\]
is well-defined and finite for any $x\in\Omega$. Also, $0\leq u\leq\overline{u}$ in $\Omega$
and since $\overline{u}\in L^1(\Omega)$ by Lemma \ref{U-L1}, then $u\in L^1(\Omega)$.

\paragraph*{\it Step 2: $u\in C(\Omega)$.}
Fix any compact $D\subset\Omega$ and choose a $c>0$ for which $\delta(x)>2c$ for any $x\in D$.
Let $\widetilde{D}:=\{y\in\Omega:\delta(y)>c\}$. For any $k,\,j\in\N$ it holds
\[
\Ds(u_{k+j}-u_k)=f(u_k)-f(u_{k+j})\leq 0, \qquad\hbox{ in }\widetilde{D}
\]
and therefore
\[
0\leq u_{k+j}(x)-u_k(x)\leq\int_{\C\widetilde{D}}P_{\widetilde{D}}(x,y)\left[u_{k+j}(y)-u_k(y)\right]dy
\]
where $P_{\widetilde{D}}(x,y)$ is the Poisson kernel associated to $\widetilde{D}$, which satisfies (see \cite[Theorem 2.10]{chen})
\[
P_{\widetilde{D}}(x,y)\ \leq\ \frac{C\,{\delta_{\widetilde{D}}(x)}^s}{{\delta_{\widetilde{D}}(y)}^s\,|x-y|^N},
\qquad x\in\widetilde{D},\ y\in\C\widetilde{D}.
\]
When $x\in D\subset\widetilde{D}$ one has $|x-y|>c$ for any $y\in\C\widetilde{D}$, and therefore
\[
0\leq u_{k+j}(x)-u_k(x)\leq C\int_{\C\widetilde{D}}\frac{u_{k+j}(y)-u_k(y)}{{\delta_{\widetilde{D}}(y)}^s}\;dy
\leq C\int_{\C\widetilde{D}}\frac{u(y)-u_k(y)}{{\delta_{\widetilde{D}}(y)}^s}\;dy
\]
where the last integral converges by Monotone Convergence to $0$ independently on $x$.
This means the convergence $u_k\to u$ is uniform on compact subsets 
and since $\{u_k\}_{k\in\N}\subset C(\Omega)$ (cf. \cite[Theorem 1.2.12]{a}),
then also $u\in C(\Omega)$.
 
\paragraph*{\it Step 3: $u\in C^2(\Omega)$.}
This is a standard bootstrap argument using the elliptic regularity in \cite[Propositions 2.8 and 2.9]{silvestre}.

\paragraph*{\it Step 4: $u$ solves \eqref{problem} in a pointwise sense.} The function $\Ds u(x)$ is well-defined 
for any $x\in\Omega$ because $u\in C^2(\Omega)\cap L^1(\R^N)$. Using the regularity results in \cite[Propositions 2.8 and 2.9]{silvestre},
we have
\[
\Ds u=\lim_{k\uparrow+\infty}\Ds u_k=-\lim_{k\uparrow+\infty}f(u_k)=-f(u).
\]
Also, $\delta^{1-s}u\geq\delta^{1-s}u_k$ holds in $\Omega$ for any $k\in\N$. Therefore, for any $k\in\N$,
\[
\liminf_{x\to\partial\Omega}\delta(x)^{1-s}u(x)\geq\lim_{x\to\partial\Omega}\delta(x)^{1-s}u_k(x)\geq \lambda Eu_k=\lambda k
\]
for some constant $\lambda>0$ depending on $\Omega$ and not on $k$. This entails
\[
\lim_{x\to\partial\Omega}\delta(x)^{1-s}u(x)\ =\ +\infty
\]
and completes the proof of Theorem \ref{main1}.

\begin{rmk}\rm The proof of Theorem \ref{main2} is alike. Indeed,
in the same way, the sequence of solutions to
\begin{equation}\label{gkproblem}
\left\lbrace\begin{aligned}
\Ds u_k\  &=\ -f(u_k) & & \hbox{ in }\Omega \\
u_k\ &=\ g_k:=\min\{k,g\} &  & \hbox{ in }\C\Omega,\ k\in\N \\
E\,u_k\ &=\ 0 & & \hbox{ on }\partial\Omega, & 
\end{aligned}\right.
\end{equation}
approaches a solution of problem \eqref{gproblem}
which lies below the supersolution provided by Theorem \ref{U-gsuper}.
\end{rmk}

\subsection{Proof of Theorem \ref{nonexist}}

Following \cite{a}, we write the Green representation for $u_k$:
\[
u_k(x)\ =\ k\int_{\partial\Omega}M_\Omega(x,\theta)\;d\sigma(\theta)
-\int_\Omega G_\Omega(x,y)\,f(u_k(y))dy,\qquad x\in\Omega.
\]
Denoting simply by
\[
h_1(x)\ :=\ \int_{\partial\Omega}M_\Omega(x,\theta)\;d\sigma(\theta),
\qquad \xi(x)\ :=\ \int_\Omega G_\Omega(x,y)\;dy,
\]
we get 
\begin{equation}\label{'}
u_k(x)\ \geq\ kh_1(x)-a\xi(x)-b\int_\Omega G_\Omega(x,y)\,u_k(y)\;dy
\ \geq\ kh_1(x)-a\xi(x)-bk\int_\Omega G_\Omega(x,y)\,h_1(y)\;dy.
\end{equation}
Recall that $\xi\asymp\delta^s$ and $h_1\asymp\delta^{s-1}$.
Applying \cite[Proposition 3]{a}, we see that
\[
h_1(x)-b\int_\Omega G_\Omega(x,y)\,h_1(y)\;dy\ >\ 0
\]
holds when $x$ is taken close enough to $\partial\Omega$.
This concludes the proof.
\hfill$\square$

\section{the power case: proof of theorem \ref{power-th}}\label{power-sec}

{\it Proof of 1.}
We show how problem 
\begin{equation}
\left\lbrace\begin{aligned}
\Ds u_1\  &=\ -u_1^p & & \hbox{ in }\Omega \\
u_1\ &=\ 0 & & \hbox{ in }\C\Omega \\
E\,u_1\ &=\ 1 & & \hbox{ on }\partial\Omega
\end{aligned}\right.
\end{equation} 
does not admit any weak or pointwise solution.

In both cases the solution would satisfy $u_1\geq c\delta^{s-1}$ in $\Omega$, for some $c>0$.
If $u_1$ was a weak solution then for any $\phi\in\T(\Omega)$
\[
\int_\Omega u_1\,\Ds\phi+\int_\Omega u_1^p\phi\ =\ \int_{\partial\Omega}D_s\phi
\]
where 
\[
\int_\Omega u_1^p\phi \geq C\int_\Omega \delta^{p(s-1)}\delta^s = +\infty
\]
because \eqref{E} does not hold, a contradiction.

If $u_1$ was a pointwise solution, then by Lemma \ref{point-weak} it 
would be a weak solution on any subdomain $D\subset\overline{D}\subset\Omega$.
Therefore
\[
u_1(x)\ =\ -\int_D G_D(x,y)\,u_1(y)^p\;dy + \int_{\C D} P_D(x,y)\,u_1(y)\;dy.
\]
If $u_0$ denotes the $s$-harmonic function induced by $Eu=1$, then $u_1\leq u_0$ in $\Omega$
and
\[
u_1(x)\ \leq\ -\int_D G_D(x,y)\,u_1(y)^p\;dy + \int_{\C D} P_D(x,y)\,u_0(y)\;dy
=-\int_D G_D(x,y)\,u_1(y)^p\;dy + u_0(x).
\]
Fix $x\in\Omega$.
Letting now $D\nearrow\Omega$ we have that $G_D(x,y)\uparrow G_\Omega(x,y)$
and 
\[
\int_\Omega G_\Omega(x,y)\,u_1(y)^p\;dy\ \geq\ 
c\delta(x)^s\int_{\{2\delta(y)<\delta(x)\}} \delta(y)^s\,u_1(y)^p\;dy =+\infty
\]
because \eqref{E} does not hold, a contradiction.
\hfill$\square$
\medskip

\noindent{\it Proof of 2.}
We apply Theorem \ref{main1} when $f(t)=t^p$.
In this case
\[
\frac{tf'(t)}{f(t)}=p>1
\]
so that hypothesis \eqref{tech} is fulfilled.
The function $\phi$ reads as (cf. equation \eqref{phi})
\[
\phi(u)=\int_u^{+\infty}\sqrt{\frac{p+1}{2}}\:t^{-\frac{p+1}{2}}\;dt=\sqrt{\frac{2(p+1)}{p-1}}\:u^{\frac{1-p}{2}}
\]
and hypothesis \eqref{L1} can then be written
\[
\int_u^{+\infty}\eta^{\frac{1-p}{2s}}\;d\eta\ <\ +\infty
\]
that holds if and only if $p>1+2s$.
On the other hand hypothesis \eqref{E} becomes
\[
p-\frac2{1-s}<-1,\qquad\hbox{i.e.}\ \ p<\frac{1+s}{1-s}=1+\frac{2s}{1-s}.
\]
\hfill$\square$
\begin{rmk}\rm We retrieve in this case some of the results in \cite[Theorem 1.1, equations $(1.6)$ and $(1.7)$]{chen-felmer}
and we obtain the explicit value of the parameter denoted\footnote{In the notations of \cite{chen-felmer},
$\alpha\in (0,1)$ is the power of the Laplacian, which corresponds to $s$ in our notations.} by $\tau_0(\alpha)\in(-1,0)$ by the authors,
which is $\tau_0(\alpha)=\alpha-1$.
\end{rmk}
\medskip

\noindent{\it Proof of 3.}
Following \cite{a}, we write the Green representation for $u_k$:
\[
u_k(x)\ =\ k\int_{\partial\Omega}M_\Omega(x,\theta)\;d\sigma(\theta)
-\int_\Omega G_\Omega(x,y)\,{u_k(y)}^pdy,\qquad x\in\Omega.
\]
Denoting simply by
\[
h_1(x)\ :=\ \int_{\partial\Omega}M_\Omega(x,\theta)\;d\sigma(\theta),
\]
we have $u_k\leq kh_1$ in $\Omega$ and
\begin{equation}\label{''}
u_k(x)\ \geq\ kh_1(x)-k^{s}\int_\Omega G_\Omega(x,y)\,u_k(y)^{p-s}\,h_1(y)^s\;dy.
\end{equation}
Define
\[
\xi(x)\ :=\ \int_\Omega G_\Omega(x,y)\;dy
\]
and recall that $\xi\asymp\delta^s$, while $h_1\asymp\delta^{s-1}$.
By \eqref{''} we deduce
\[
\int_\Omega u_k\ \geq\ k\int_\Omega h_1-\int_\Omega u_k^{p-s}\,h_1^s\,\xi.
\]
Since $p\in(1,1+s)$, it holds $p-s\in(1-s,1)$ and $u_k^{p-s}\leq u_k$.
Thus there exists some constant $C>0$ for which
\[
\int_\Omega u_k+C\int_\Omega u_k\,\delta^{s(s-1)+s}\ \geq\ k\int_\Omega h_1
\]
where $s(s-1)+s>0$, so (modifying $C$ if necessary)
\[
(C+1)\int_\Omega u_k\ \geq\ k\int_\Omega h_1
\]
which concludes the proof.
\hfill$\square$
\medskip

\noindent{\it Proof of 4.}
This is a straightforward consequence of Theorem \ref{nonexist}.
\hfill$\square$
\medskip

\section{remarks and comments}

In this section we would like to point out some elements
that may risk to be unclear if left implicit.
In the first paragraph we discuss the relation between pointwise solutions
and weak $L^1$ solutions. The second paragraph deals with the definition
of weak $L^1$ solution given by Chen and V\'eron \cite{chen-veron},
which amounts to be equivalent to the one given in \cite{a}.
The remainder of the section will be devoted to
the explanation of the difficulties of problem \eqref{problem} when 
one of hypotheses \eqref{L1} or \eqref{E} fails.

\subsection{Pointwise solutions vs. weak \texorpdfstring{$L^1$}{L1} solutions}

For the sake of clarity we recall here the definitions involved. In the following
$\Omega$ will always be a bounded open subset of $\R^N$ with $C^2$ boundary.
\begin{defi} Given three measurable functions
\[
f:\Omega\longrightarrow\R,\qquad g:\C\Omega\longrightarrow\R, \qquad h:\partial\Omega\longrightarrow\R
\]
a function $u:\R^N\longrightarrow\R$ is said to be a p\underline{ointwise solution} of
\[
\left\lbrace\begin{aligned}
\Ds u &=f & & \hbox{ in }\Omega \\
u &=g & & \hbox{ in }\C\Omega \\
Eu &=h & & \hbox{ on }\partial\Omega 
\end{aligned}\right.
\]
provided
\begin{enumerate}
\item $u\in L^1(\Omega)$
\item for any $x\in\C\Omega$ it holds $u(x)=g(x)$
\item the principal value
\[
PV\int_{\R^N}\frac{u(x)-u(y)}{|x-y|^{N+2s}}\;dy
\]
converges for any $x\in\Omega$ and
\[
\A(N,s)\,PV\int_{\R^N}\frac{u(x)-u(y)}{|x-y|^{N+2s}}\;dy\ =\ f(x),\qquad \hbox{for any } x\in\Omega
\]
\item for any $\theta\in\partial\Omega$ there exists the limit $\lim_{x\to\theta}\delta(x)^{1-s}u(x)$ and the
renormalized limit $Eu$ satisfies $Eu(\theta)=h(\theta)$.
\end{enumerate}
\end{defi}

\begin{defi}\label{weakdefi1} Given three measurable functions
\[
f:\Omega\longrightarrow\R,\qquad g:\C\Omega\longrightarrow\R, \qquad h:\partial\Omega\longrightarrow\R
\]
a function $u:\R^N\longrightarrow\R$ is said to be a \underline{weak $L^1$ solution} of
\[
\left\lbrace\begin{aligned}
\Ds u &=\ f &\hbox{ in }\Omega \\
u &=\ g &\hbox{ in }\C\Omega \\
Eu &=\ h &\hbox{ on }\partial\Omega 
\end{aligned}\right.
\]
provided $u\in L^1(\Omega)$ and for any $\phi\in\T(\Omega)=\{\phi\in C^s(\R^N):\Ds\phi|_\Omega\in C^\infty_c(\Omega),\ 
\phi=0 \hbox{ in }\C\Omega\}$ the following holds
\[
\int_\Omega u\Ds\phi\ =\ \int_\Omega f\phi - \int_{\C\Omega} g\Ds\phi + \int_{\partial\Omega} h\,D_s\phi.
\]
\end{defi}
For further details and notation, we refer to \cite{a}.

\begin{lem}\label{weak-point} Take $f\in C^\alpha_{loc}(\Omega)$ for some $\alpha\in(0,1)$ with
\[
\int_\Omega |f|\delta^s\ <\ +\infty,
\]
$g:\C\Omega\to\R$ measurable with
\[
\int_{\C\Omega}|g|\delta^{-s}\min\{1,\delta^{-N-s}\}\ <\ +\infty,
\]
$h\in C(\partial\Omega)$ and $u:\R^N\to\R$ a weak $L^1$ solution to
\[
\left\lbrace\begin{aligned}
\Ds u &= f & & \hbox{ in }\Omega \\
u &= g & & \hbox{ in }\C\Omega \\
Eu &= h & & \hbox{ on }\partial\Omega.
\end{aligned}\right.
\]
Then $u$ is also a pointwise solution.
\end{lem}
\begin{proof}
We can write $u$ as the sum
\[
u(x)\ =\ \int_\Omega G_\Omega(x,y)\,f(y)\;dy +u_0(x)
\]
where $u_0$ is the $s$-harmonic function induced in $\Omega$ by the data $g$ and $h$.
For any $x\in\Omega$ it holds in a pointwise sense $\Ds u(x)=f(x)$ in view of the regularity of $f$
and the construction of the Green kernel. Then, to completely prove the lemma, it suffices to prove 
\[
\lim_{x\to\partial\Omega}\delta(x)^{1-s}\int_\Omega G_\Omega(x,y)\,f(y)\;dy\ =\ 0.
\]
This is proved in Lemma \ref{E0} below.
\end{proof}

\begin{lem}\label{point-weak} Take $f\in C^\alpha_{loc}(\Omega)$ for some $\alpha\in(0,1)$, $h\in C(\partial\Omega)$ and $u:\R^N\to\R$
a pointwise solution to
\begin{equation}\label{98}
\left\lbrace\begin{aligned}
\Ds u &= f & & \hbox{ in }\Omega \\
u &= g & &\hbox{ in }\C\Omega \\
Eu &= h & &\hbox{ on }\partial\Omega.
\end{aligned}\right.
\end{equation}
If
\[
\int_\Omega |f|\delta^s\ <\ +\infty,\qquad \int_{\C\Omega}|g|\delta^{-s}\min\{1,\delta^{-N-s}\}\ <\ +\infty,
\qquad h\in C(\partial\Omega)
\]
then $u$ is also a weak $L^1$ solution to the same problem.
\end{lem}
\begin{proof}
We refer to \cite[Theorem 1.2.8]{a} for the existence and uniqueness of a weak $L^1$ solution $v$ to problem \eqref{98}.
By Lemma \ref{weak-point}, $v$ is also a pointwise solution. Thus 
\begin{equation*}
\left\lbrace\begin{aligned}
\Ds (u-v) &=\ 0 &\hbox{ in }\Omega \\
u-v &=\ 0 &\hbox{ in }\C\Omega \\
E(u-v) &=\ 0 &\hbox{ on }\partial\Omega.
\end{aligned}\right.
\end{equation*}
in a pointwise sense. In particular, $u-v\in C(\Omega)$ since harmonic functions are smooth.
Define $\Omega^+:=\{x\in\Omega:u(x)>v(x)\}$: $u-v$ is a nonnegative $s$-harmonic function
and, by \cite[Lemma 5 and Theorem 1]{bogdan}, it decomposes into the sum of the $s$-harmonic function
induced by the $E_{\Omega^+}(u-v)$ trace and the one by its values on $\C\Omega^+$.
But, on the one hand $E_{\Omega^+}(u-v)=0$ on $\partial\Omega^+$ as it is implied 
by the singular trace datum in \eqref{98} and the continuity on $\partial\Omega^+\cap\Omega$,
while $u-v\leq 0$ in $\C\Omega^+$. This yields $\Omega^+=\emptyset$ and $v\geq u$ in $\Omega$.
Repeting the argument we deduce also $u\leq v$ and this completes the proof of the lemma.
\end{proof}

\begin{lem}\label{E0}
Let $f:\Omega\to\R$ be a continuous function such that
\begin{equation}\label{564}
\int_\Omega|f|\,\delta^s\ <\ +\infty.
\end{equation}
Then 
\begin{equation}\label{weaktrace}
\lim_{\eta\downarrow0}\left(\frac1\eta\int_{\{\delta(x)<\eta\}}\delta(x)^{1-s}\int_\Omega G_\Omega(x,y)\,f(y)\;dy\;dx\right)\ =\ 0.
\end{equation}
\end{lem}
\begin{proof}
Equation \eqref{weaktrace} expresses a notion of weak trace at the boundary 
introduced by Ponce \cite[Proposition 3.5]{ponce}.
Choose $\eta>0$ small and consider the integral
\begin{equation}\label{7483}
\frac1\eta\int_\Omega\delta(x)^{1-s}\chi_{(0,\eta)}(\delta(x))\int_\Omega G_\Omega(x,y)\,f(y)\;dy\;dx.
\end{equation}
We are going to show that it converges to $0$ as $\eta\downarrow 0$.
By splitting $f$ into its positive and negative part,
we can suppose $f\geq 0$ without loss of generality.
Fix $\sigma\in(0,s)$ and exchange the order of integration in \eqref{7483}. Our claim is that
\begin{equation}\label{768}
\int_\Omega G_\Omega(x,y)\,\delta(x)^{1-s}\chi_{(0,\eta)}(\delta(x))\;dx\ \leq\ 
\left\lbrace\begin{aligned}
& C\eta^{1+\sigma}\delta(y)^{s-\sigma} & \hbox{ if }\delta(y)\geq\eta \\
& C\eta\delta(y)^s & \hbox{ if }\delta(y)<\eta.
\end{aligned}\right.
\end{equation}
This would prove
\begin{multline*}
\frac1\eta\int_\Omega f(y)\int_\Omega G_\Omega(x,y)\,\delta(x)^{1-s}\chi_{(0,\eta)}(\delta(x))\;dx\;dy\ \leq \\
\leq\ C\eta^\sigma\int_{\{\delta(y)\geq\eta\}\cap\Omega}f(y)\delta(y)^{s-\sigma}\;dy+C\int_{\{\delta(y)<\eta\}\cap\Omega}f(y)\delta(y)^s\;dy
\end{multline*}
where the second addend converges to $0$ as $\eta\downarrow 0$ by \eqref{564}.
For the first addend, we have that $\eta^\sigma f(y)\delta(y)^{s-\sigma}$ converges pointwisely to zero
for any $y\in\Omega$ and $\eta^\sigma f(y)\delta(y)^{s-\sigma}\leq f(y)\delta(y)^s$
if $y\in\Omega\cap\{\delta(y)>\eta\}$, therefore we have the convergence to $0$ by dominated convergence.

We turn now to the proof of \eqref{768}. For any $y\in\Omega$ one has 
\begin{equation}\label{654}
\int_\Omega G_\Omega(x,y)\,\delta(x)^{1-s}\chi_{(0,\eta)}(\delta(x))\;dx\leq
\eta^{1+\sigma}\int_\Omega G_\Omega(x,y)\,\delta(x)^{-s-\sigma}\;dx\leq
C\eta^{1+\sigma}\delta(y)^{s-\sigma}
\end{equation}
where we have used the regularity at the boundary in \cite[Proposition 1.2.9]{a}.
In particular \eqref{654} holds when $\delta(y)>\eta$. 

To prove the other part of \eqref{768} we write 
(dropping from now on multiplicative constants depending on $N,\Omega$ and $s$)
\[
\int_\Omega G_\Omega(x,y)\,\delta(x)^{1-s}\chi_{(0,\eta)}(\delta(x))\;dx
\leq\ \eta^{1-s}\int_{\{\delta(x)<\eta\}\cap\Omega}\frac{\left(\delta(x)\delta(y)\wedge|x-y|^2\right)^s}{{|x-y|}^N}\delta(x)^{1-s}dx
\]
and we are allowed to perform 
the computations only in the case where $\partial\Omega$ is locally flat
where the above reads as
\[
\int_0^\eta\int_B\frac{\left[x_Ny_N\wedge(|x'-y'|^2+|x_N-y_N|^2)\right]^s}{{(|x'-y'|^2+|x_N-y_N|^2)}^{N/2}}\cdot x_N^{1-s}\;dx'\:dx_N.
\]
where $x=(x',x_N)\in\R^{N-1}\times\R$ and $y=(y',y_N)\in\R^{N-1}\times\R$.
First note that we can suppose $y'=0$ without loss of generality 
and $a\wedge b\leq 2ab/(a+b)$ when $a,b>0$. With the change of variable 
$x_N=y_N\,t$ and $x'=y_N\,\xi$, we reduce to
\[
y_N^{1+s}\int_0^{\eta/y_N}\int_{B_{1/y_N}}\frac{t}{\left(|\xi|^2+|t-1|^2\right)^{N/2-s}}
\cdot\frac{d\xi}{\left(|\xi|^2+|t-1|^2+t\right)^s}\;dt
\]
and, passing to polar coordinates in the $\xi$ variable
\begin{multline*}
y_N^{1+s}\int_0^{\eta/y_N}\int_0^{1/y_N}\frac{t\,\rho^{N-2}}{\left(\rho^2+|t-1|^2\right)^{N/2-s}}
\cdot\frac{d\rho}{\left(\rho^2+|t-1|^2+t\right)^s}\;dt\ \leq\\
\leq\ y_N^{1+s}\int_0^{\eta/y_N}\int_0^{1/y_N}\frac{t\,\rho}{\left(\rho^2+|t-1|^2\right)^{3/2-s}}
\cdot\frac{d\rho}{\left(\rho^2+|t-1|^2+t\right)^s}\;dt.
\end{multline*}
We deal first with the integral in the $\rho$ variable\footnote{The
computation which follows is not valid in the particular case $s=1/2$,
but with some minor natural modifications the same idea will work.}
\begin{align*}
& t\int_0^{1/y_N}\frac{\rho}{\left(\rho^2+|t-1|^2\right)^{3/2-s}}
\cdot\frac{d\rho}{\left(\rho^2+|t-1|^2+t\right)^s}\ \leq \\
& \leq\ \frac{t}{\left(|t-1|^2+t\right)^s}\int_0^1\frac{\rho}{\left(\rho^2+|t-1|^2\right)^{3/2-s}}\;d\rho+
t\int_1^{1/y_N}\frac{\rho}{\left(\rho^2+|t-1|^2\right)^{3/2}}\;d\rho \\
& \leq\ \frac{t}{\left(|t-1|^2+t\right)^s}\cdot\left.\frac{\left(\rho^2+|t-1|^2\right)^{-1/2+s}}{2s-1}\right|_{\rho=0}^1+
\frac{t}{\left(1+|t-1|^2\right)^{1/2}}.
\end{align*}
We then have
\begin{multline*}
t\int_0^{1/y_N}\frac{\rho}{\left(\rho^2+|t-1|^2\right)^{3/2-s}}
\cdot\frac{d\rho}{\left(\rho^2+|t-1|^2+t\right)^s}\ \leq\\
\leq\ \left\lbrace\begin{aligned}
& \frac{t}{\left(|t-1|^2+t\right)^s\,\left|t-1\right|^{1-2s}}+\frac{t}{\left(1+|t-1|^2\right)^{1/2}} & s\in(0,1/2), \\
& \frac{t\,\left(1+|t-1|^2\right)^{s-1/2}}{\left(|t-1|^2+t\right)^s}+\frac{t}{\left(1+|t-1|^2\right)^{1/2}} & s\in(1/2,1).
\end{aligned}\right.
\end{multline*}
The two quantities are both integrable in $t=1$
and converge to a positive constant as $t\uparrow+\infty$, therefore
\[
y_N^{1+s}\int_0^{\eta/y_N}\int_0^{1/y_N}\frac{t\,\rho^{N-2}}{\left(\rho^2+|t-1|^2\right)^{N/2-s}}
\cdot\frac{d\rho}{\left(\rho^2+|t-1|^2+t\right)^s}\;dt\ \leq\ 
\eta\,y_N^s=\eta\,\delta(y)^s
\]
completing the proof of \eqref{768}.
\end{proof}

\subsection{Variational weak solutions vs. weak \texorpdfstring{$L^1$}{L1} solutions}

In this paragraph we are going to prove the equivalence - for some class of Dirichlet problems -
between the definition of weak $L^1$ solution and the more standard one 
of variational weak solution.

\begin{defi} Given $f\in L^\infty(\Omega)$, a variational weak solution of
\begin{equation}\label{345435345}
\left\lbrace\begin{aligned}
\Ds u &= f & & \hbox{in }\Omega \\
u &= 0 & & \hbox{in }\C\Omega \\
\end{aligned}\right.
\end{equation}
is a function $u\in H^s(\R^N)$ such that $u\equiv 0$ in $\C\Omega$ and
for any other $v\in H^s(\R^N)$ such that $v\equiv 0$ in $\C\Omega$ it holds
\[
\int_{\R^N}{(-\lapl)}^{s/2}u\,{(-\lapl)}^{s/2}v\ =\ \int_\Omega f\,v.
\]
\end{defi}

\begin{lem}\label{644351} Recall the definition of the space $\T(\Omega)$ given in Definition \ref{weakdefi1}.
It holds $\T(\Omega)\subset H^s(\R^N)$.
\end{lem}
\begin{proof}
Consider $\phi\in\T(\Omega)$. The fractional Laplacian ${(-\lapl)}^{s/2}\phi$
is a continuous function decaying like $|x|^{-N-s}$ at infinity.
So $\|{(-\lapl)}^{s/2}\phi\|_{L^2(\R^N)}<\infty$ and we can apply \cite[Proposition 3.6]{hitchhiker}.
\end{proof}

\begin{prop} Let $f\in L^\infty(\Omega)$. Let $u$ be a variational
weak solution of \eqref{345435345}. Then it is also a weak $L^1$ solution
to the problem
\[
\left\lbrace\begin{aligned}
\Ds u &= f & & \hbox{in }\Omega \\
u &= 0 & & \hbox{in }\C\Omega \\
Eu &= 0 & & \hbox{on }\partial\Omega.
\end{aligned}\right.
\]
\end{prop}
\begin{proof}
Consider $\phi\in\T(\Omega)$:
\[
\int_\Omega u\Ds\phi=\int_{\R^N}{(-\lapl)}^{s/2}u\,{(-\lapl)}^{s/2}\phi=\int_\Omega f\phi.
\]
where we have used Lemma \ref{644351} on $\phi$.
\end{proof}

\begin{prop} Let $f\in L^\infty(\Omega)$. Let $u$ be a weak $L^1$ solution
to the problem
\[
\left\lbrace\begin{aligned}
\Ds u &= f & & \hbox{in }\Omega \\
u &= 0 & & \hbox{in }\C\Omega \\
Eu &= 0 & & \hbox{on }\partial\Omega.
\end{aligned}\right.
\]
Then it is also a variational
weak solution of \eqref{345435345}. 
\end{prop}
\begin{proof}
Call $w$ the variational weak solution of \eqref{345435345}. 
By the previous Lemma, $w$ is also a weak $L^1$ solution.
We must conclude $u=w$ by the uniqueness of a weak $L^1$ solution.
\end{proof}

\subsection{The test function space}

In \cite{chen-veron} the following definition of weak solution is given.

\begin{defi}\label{weakdefi2} Given a Radon measure $\nu$ such that $\delta^s\in L^1(\Omega,d\nu)$
a function $u\in L^1(\Omega)$ is a weak solution of
\[
\left\lbrace\begin{aligned}
\Ds u+ f(u) &=\ \nu & \hbox{ in }\Omega \\
u &=\ 0 & \hbox{ in }\C\Omega
\end{aligned}\right.
\]
if $f(u)\in L^1(\Omega,\delta^s\,dx)$
\[
\int_\Omega u\Ds\xi+\int_\Omega f(u)\xi\ =\ \int_\Omega \xi\;d\nu
\] 
for any $\xi\in\mathbb{X}_s\subset C(\R^N)$, i.e. 
\begin{enumerate}
\item supp$\xi\subseteq\overline{\Omega}$
\item $\Ds\xi(x)$ is pointwisely defined for any $x\in\Omega$ and $\|\Ds\xi\|_{L^\infty(\Omega)}<+\infty$ 
\item there exist a positive $\varphi\in L^1(\Omega,\delta^s dx)$ and $\eps_0>0$ such that
\[
|\Ds_\eps\xi(x)|=\left|\int_{\C B_\eps(x)}\frac{\xi(x)-\xi(y)}{|x-y|^{N+2s}}\;dx\right|\leq \varphi(x)
\qquad \hbox{ for all }\eps\in(0,\eps_0].
\]
\end{enumerate}
\end{defi}

The test space $\mathbb{X}_s$ in Definition \ref{weakdefi2} is quite different 
from the space $\T(\Omega)$ which is used in Definition \ref{weakdefi1}.
Still, testing a Dirichlet problem against one or the other does not yield two different solutions,
i.e. the two notions of weak $L^1$ solutions are equivalent.
We split the proof of this fact into two lemmata.

\begin{lem} $\T(\Omega)\subset\mathbb{X}_s$.
\end{lem}
\begin{proof} Pick $\phi\in\T(\Omega)$.
Properties {\it 1.} and {\it 2.} of Definition \ref{weakdefi2} are satisfied by construction.
In order to prove {\it 3.} write for $\delta(x)<2\eps$
\begin{multline}\label{111}
\Ds_\eps\phi(x)=\psi(x)-PV\int_{B_\eps(x)}\frac{\phi(x)-\phi(y)}{|x-y|^{N+2s}}\;dy\ =\\
=\ \psi(x)-PV\int_{B_{\delta(x)/2}(x)}\frac{\phi(x)-\phi(y)}{|x-y|^{N+2s}}\;dy
-\int_{B_\eps(x)\setminus B_{\delta(x)/2}(x)}\frac{\phi(x)-\phi(y)}{|x-y|^{N+2s}}\;dy.
\end{multline}
with $\psi:=\Ds\phi|_\Omega\in C^\infty_c(\Omega)$.
Consider $\alpha\in(0,s)$. For the first integral
\begin{align*}
& \left|PV\int_{B_{\delta(x)/2}(x)}\frac{\phi(x)-\phi(y)}{|x-y|^{N+2s}}\;dy\right|\ 
 \leq\ \|\phi\|_{C^{2s+\alpha}(B_{\delta(x)/2}(x))}\int_{B_{\delta(x)/2}(x)}\frac{dy}{|x-y|^{N-\alpha}}\ \\
& =\ \|\phi\|_{C^{2s+\alpha}(B_{\delta(x)/2}(x))}\frac{\omega_{N-1}}{\alpha}\left(\frac{\delta(x)}{2}\right)^\alpha
\end{align*}
where, by \cite[Proposition 2.8]{silvestre}
\begin{align*}
& \|\phi\|_{C^{2s+\alpha}(B_{\delta(x)/2}(x))}=2^{2s+\alpha}\delta(x)^{-2s-\alpha}\left\|\phi\left(x+\frac{\delta(x)}{2}\ \cdot\right)\right\|_{C^{2s+\alpha}(B)}\ \leq \\
& \leq\ C\delta(x)^{-2s-\alpha}\left(\left\|\phi\left(x+\frac{\delta(x)}{2}\ \cdot\right)\right\|_{L^\infty(B)}+
\left\|\psi\left(x+\frac{\delta(x)}{2}\ \cdot\right)\right\|_{C^{\alpha}(B)}\right) \\
& \leq\ C\delta(x)^{-2s-\alpha}\left(\|\phi\|_{L^\infty(B_{\delta(x)/2}(x))}
+\delta(x)^\alpha\|\psi\|_{C^{\alpha}(B_{\delta(x)/2}(x))}\right) \\
& \leq\ C\delta(x)^{-2s-\alpha}\left(\|\psi\|_{L^\infty(\R^N)}\delta(x)^s
+\delta(x)^\alpha\|\psi\|_{C^{\alpha}(\R^N)}\right)\ 
\leq\ C\|\psi\|_{C^\alpha(\R^N)}\delta(x)^{-2s}.
\end{align*}
The integration far from $x$ gives
\begin{align*}
& \left|\int_{B_\eps(x)\setminus B_{\delta(x)/2}(x)}\frac{\phi(x)-\phi(y)}{|x-y|^{N+2s}}\;dy\right|\ 
 \leq\ \|\phi\|_{C^s(\R^N)} \int_{B_\eps(x)\setminus B_{\delta(x)/2}(x)}\frac{dy}{|x-y|^{N+s}}\ \leq \\ 
& \leq\ \|\phi\|_{C^s(\R^N)} \int_{\R^N\setminus B_{\delta(x)/2}}\frac{dz}{|z|^{N+s}}\ 
 \leq\ \|\phi\|_{C^s(\R^N)} \frac{\omega_{N-1}}{s}\left(\frac2{\delta(x)}\right)^s.
\end{align*}
All this entails
\[
\delta(x)^s|\Ds_\eps\phi(x)|\ \leq\ \delta(x)^s|\psi(x)|+C\|\psi\|_{C^\alpha(\R^N)}\delta(x)^{\alpha-s}+C\|\phi\|_{C^s(\R^N)},
\qquad \hbox{when }\delta(x)<2\eps.
\]
For $\delta(x)\geq2\eps$ one does not have the second integral on the right-hand side of \eqref{111}
whereas the first one is computed on the ball of radius $\eps$,
but the same computations can be run.
This proves the statement of the Lemma.
\end{proof}

\begin{lem} Given a Radon measure $\nu\in\mathcal{M}(\Omega)$ such that $\delta^s\in L^1(\Omega,d\nu)$,
if a function $u\in L^1(\Omega)$ satisfies 
\begin{equation}\label{7890}
\int_\Omega u\Ds\xi\ =\ \int_\Omega\xi\;d\nu,\qquad\hbox{ for any }\xi\in\T(\Omega),
\end{equation}
then the same holds true for any $\xi\in\mathbb{X}_s$. 
\end{lem}
\begin{proof}
Pick $\xi\in\mathbb{X}_s$: by definition, $\zeta:=\Ds\xi\in L^\infty(\Omega)$. 
Consider the standard mollifier $\eta\in C^\infty_c(\R^N)$ and $\eta_\eps(x):=\eps^{-N}\eta(x/\eps)$.
Then
\begin{equation}\label{zetaeps}
\zeta_\eps:=\zeta\chi_\Omega*\eta_\eps \in C^\infty(\R^N)\hbox{ and }\|\zeta_\eps\|_{L^\infty(\Omega)}\leq\|\zeta\|_{L^\infty(\Omega)}.
\end{equation}
Define $\xi_\eps$ as the solution to
\[
\left\lbrace\begin{aligned}
\Ds\xi_\eps &=\ \zeta_\eps & \hbox{ in }\Omega \\
\xi_\eps &=\ 0 & \hbox{ in }\C\Omega \\
E\xi_\eps &=\ 0 & \hbox{ on }\partial\Omega.
\end{aligned}\right.
\]
Also, for $\rho>0$ small consider
\[
\Omega_\rho:=\{x\in\Omega:\delta(x)>\rho\}
\]
and a bump function $b_\rho\in C^\infty_c(\R^N)$ such that
\[
b_\rho\equiv 1\hbox{ in }\Omega_{2\rho},\ b_\rho\equiv0\hbox{ in }\R^N\setminus\Omega_\rho,\ 0\leq b_\rho\leq 1\hbox{ in }\R^N.
\]
Then $\zeta_{\eps,\rho}:=b_\rho\zeta_\eps\in C^\infty_0(\Omega)$. 
Let $\xi_{\eps,\rho}\in\T(\Omega)$ be the function induced by $\zeta_{\eps,\rho}$.
By \eqref{7890},
\begin{equation}\label{231456}
\int_\Omega u\zeta_{\eps,\rho}\ =\ \int_\Omega \xi_{\eps,\rho}\;d\nu.
\end{equation}
It holds $\zeta_{\eps,\rho}\rightarrow\zeta_\eps$ as $\rho\downarrow0$,
with $\|\zeta_{\eps,\rho}\|_{L^\infty(\Omega)}\leq\|\zeta_\eps\|_{L^\infty(\Omega)}$ and
\[
|\xi_{\eps,\rho}(x)|\ \leq\ C\|\zeta_{\eps,\rho}\|_{L^\infty(\Omega)}\delta(x)^s
\leq \|\zeta_{\eps}\|_{L^\infty(\Omega)}\delta(x)^s
\]
so that we can push equality \eqref{231456} to the limit to deduce,
by dominated convergence,
\begin{equation}\label{2356}
\int_\Omega u\zeta_\eps\ =\ \int_\Omega \xi_\eps\;d\nu.
\end{equation}
Similarly, since $\|\zeta_{\eps}\|_{L^\infty(\Omega)}\leq \|\zeta\|_{L^\infty(\Omega)}$,
letting $\eps\downarrow 0$ yields
\[
\int_\Omega u\zeta\ =\ \int_\Omega \xi\;d\nu.
\]
\end{proof}

%

\section{examples}

In the next two examples we look at the two critical cases in the power-like nonlinearity,
adding a logarithmic weight.

\begin{ex}[\bf Lower critical case for powers]\rm
We consider here $f(t)=t^{1+2s}\ln^\alpha(1+t)$, $\alpha>0$.
In this case 
\[
\frac{tf'(t)}{f(t)}=\frac{(1+2s)f(t)+\frac{\alpha t f(t)}{(1+t)\ln(1+t)}}{f(t)}=1+2s+\frac{\alpha t}{(1+t)\ln(1+t)}.
\]
Condition \eqref{L1-bis} turns into
\[
\int_u^{+\infty}\left(\frac{t}{t^{1+2s}\ln^\alpha(1+t)}\right)^{1/(2s)}dt=
\int_u^{+\infty}\frac{dt}{t\ln^{\alpha/(2s)}(1+t)}\ <\ +\infty
\]
which is fulfilled only for $\alpha>2s$. Also, hypothesis \eqref{E} becomes
\[
\int_{t_0}^{+\infty}t^{1+2s-2/(1-s)}\ln^\alpha(1+t)\;dt\ <\ +\infty
\]
which is satisfied by any $\alpha>0$ since $(1+2s)(1-s)-2<s-1$.
\hfill$\blacklozenge$
\end{ex}

\begin{ex}[\bf Upper critical case for powers]\rm
We consider here $f(t)=t^{\frac{1+s}{1-s}}\ln^{-\beta}(1+t)$, $\beta>0$.
In this case 
\[
\frac{tf'(t)}{f(t)}=\frac{\frac{1+s}{1-s}f(t)-\frac{\beta t f(t)}{(1+t)\ln t}}{f(t)}=\frac{1+s}{1-s}-\frac{\beta t}{(1+t)\ln(1+t)}
\]
Hypothesis \eqref{L1-bis} turns into
\[
\int_u^{+\infty}\left(\frac{t\ln^\beta(1+t)}{t^{(1+s)/(1-s)}}\right)^{1/(2s)}\;dt=
\int_u^{+\infty}\frac{\ln^{\beta/(2s)}(1+t)}{t^{1/(1-s)}}\;dt
\ <\ +\infty
\]
which is fulfilled for any $\beta>0$. Also, hypothesis \eqref{E} becomes
\[
\int_{t_0}^{+\infty}t^{-1}\ln^{-\beta}(1+t)\;dt\ <\ +\infty
\]
which is satisfied by any $\beta>1$.
\hfill$\blacklozenge$
\end{ex}

\section*{acknowledgements}

The author is thankful to L. Dupaigne and E. Valdinoci for
the reading of the paper and fruitful discussion.

\section*{references}
\addcontentsline{toc}{section}{references}



\end{document}